\documentclass{article}
\usepackage{amsmath, amssymb, amsthm, latexsym, mathtools}
\usepackage{authblk, color}

\title{Dirichlet-Poincar\'{e} profiles of graphs and groups}
\author{David Hume\thanks{The author was supported by a Titchmarsh Fellowship of the University of Oxford}}

\affil{Mathematical Institute, University of Oxford, OX2 6GG, UK}
\date{\today}

\newtheorem{proposition}{Proposition}[section]
\newtheorem{theorem}[proposition]{Theorem}
\newtheorem{bthm}{Theorem}[]
\newtheorem{bcor}[bthm]{Corollary}
\newtheorem{bprop}[bthm]{Proposition}
\newtheorem{bqu}[bthm]{Question}
\newtheorem{bconj}[bthm]{Conjecture}

\newtheorem{corollary}[proposition]{Corollary}
\newtheorem{lemma}[proposition]{Lemma}

\theoremstyle{definition}
\newtheorem{defn}[proposition]{Definition}
\newtheorem{question}[proposition]{Question}

\newcommand{\set}[1]{\left\{#1\right\}}
\newcommand{\setcon}[2]{\left\{\left.#1\,\right|\,#2\right\}}

\newcommand{\colsetcon}[2]{\left\{#1\,:\,#2\right\}}

\newcommand{\R}{\mathbb{R}}
\newcommand{\Z}{\mathbb{Z}}
\newcommand{\N}{\mathbb{N}}

\newcommand{\norm}[1]{\left\lVert#1\right\rVert}
\newcommand{\abs}[1]{\left\lvert#1\right\rvert}

\newcommand{\tu}{\textup}

\newcommand{\diam}{\tu{diam}}

\begin{document}
\maketitle
\begin{abstract}
\noindent We define Poincar\'{e} profiles of Dirichlet type for graphs of bounded degree, in analogy with the Poincar\'{e} profiles (of Neumann type) defined in \cite{HumeMackTess1}. The obvious first definition yields nothing of interest, but an alternative definition yields a spectrum of profiles which are quasi--isometry invariants and monotone with respect to subgroup inclusion. Moreover, in the extremal cases $p=1$ and $p=\infty$, they detect the F\o lner function and the growth function respectively.
\end{abstract}

\section{Introduction}
In \cite{HumeMackTess1}, a spectrum of monotone coarse invariants for bounded degree graphs were introduced: these were called $L^p$-Poincar\'{e} profiles and are defined for $p\in[1,\infty]$. At the extremes $p=1$ and $p=\infty$, the profiles detect the separation profile (see \cite{BenSchTim-12-separation-graphs}) and the growth function of the graph respectively.

A more accurate name is $L^p$-\textbf{Neumann}--Poincar\'{e} profiles, since they are built from Poincar\'{e} constants of Neumann type.

The goal of this paper is to provide both a short and a long answer to the following question pointed out to us by Laurent Saloff--Coste:

\begin{bqu} What happens if you replace Poincar\'e constants of Neumann type by Poincar\'e constants of \textbf{Dirichlet type}?
\end{bqu}

%

Let us fix some notation. Let $\Gamma$ be a finite graph and let $f:V\Gamma\to\R$. We define $\nabla f:V\Gamma\to\R$ by $\nabla f(v)=\max\setcon{\abs{f(v)-f(w)}}{vw\in E\Gamma}$.

Given a graph $X$ and a finite subgraph $\Gamma\leq X$ we define the $L^p$-Dirichlet--Poincar\'{e} constant of $\Gamma$ as follows: 
\[
 Dh^p_X(\Gamma)=\inf\setcon{\frac{\norm{\nabla f}_p}{\norm{f}_p}}{f:V\Gamma\to\R,\ f|_{\partial_X\Gamma}\equiv 0,\ f\not\equiv 0},
\]
where $\partial_X\Gamma=\setcon{v\in V\Gamma}{d_X(v,VX\setminus V\Gamma)=1}$. We make the convention that $Dh^p_X(\Gamma)=+\infty$ whenever $\Gamma=\partial_X\Gamma$. Notice that replacing $f$ by $\abs{f}$ does not change the norm, preserves the property $f|_{\partial_X \Gamma}\equiv 0$ and does not increase $\norm{\nabla f}_p$. Therefore we may always assume that all functions we consider are non-negative. When the ambient graph $X$ is clear we will simply write $Dh^p(\Gamma)$ for $Dh^p_X(\Gamma)$.

Following \cite{HumeMackTess1} we \textit{could} define the $L^p$-Dirichlet-Poincar\'e profile of $X$ by
\[
 D^*\Lambda^p_X(n)=\sup\setcon{\abs{\Gamma}Dh^p_X(\Gamma)}{\Gamma\leq X,\ \partial_X\Gamma\neq  V\Gamma,\ \abs{V\Gamma}\leq n}.
\]
While $D^*\Lambda^p$ defines a monotone coarse invariant, it is sadly not a very interesting one, providing the short answer to the question.

\begin{bprop}\label{bprop:detectinf} Let $X$ be a connected infinite graph with maximal vertex degree $d$. Then
\[
 D^*\Lambda^p_X(n) =
 \left\{
  \begin{array}{rl}
   (d+1)^\frac1p n & \textup{if } p<\infty, \\
   n & \textup{if } p=\infty.
  \end{array}
 \right.
\]
\end{bprop}

Instead, let us make the following alternative definition. Let $X$ be a graph of bounded degree. The $L^p$-\textbf{Dirichlet-Poincar\'e profile} of $X$ is given by
\[
 D\Lambda^p_X(n) = \inf\colsetcon{\abs{\Gamma}Dh^p_X(\Gamma)}{\Gamma\leq X,\ \abs{V\Gamma}\geq n}.
\]
In general, these profiles are not monotone coarse invariants (nor even monotone under quasi-isometric embeddings), but they are quasi-isometry invariants. 

\begin{bthm}\label{thm:DirPoinQI} Let $X,Y$ be quasi-isometric graphs of bounded degree. Then, for every $p\in[1,\infty]$,
\[
 D\Lambda^p_X(n) \simeq  D\Lambda^p_Y(n).
\]
\end{bthm}

We consider functions using a standard partial order: given two functions $f,g:\N\to\R$, we write $f\lesssim g$ if there exists a constant $C$ such that $f(n)\leq Cg(Cn)+C$ for all $n\in\N$, and $f\simeq g$ if $f\lesssim g$ and $g\lesssim f$. We write $f\lesssim_{u,v,\ldots} g$ to indicate that the constant $C$ depends on $u,v,\ldots$.
\medskip
Hence, Dirichlet-Poincar\'e profiles are well-defined for finitely generated groups. Moreover, they behave monotonically with respect to subgroups.

\begin{bthm}\label{thm:Dirsubgp} Let $H$ be a finitely generated subgroup of a finitely generated group $G$. Then, for every $p\in[1,\infty]$,
\[
 D\Lambda^p_H(n) \lesssim  D\Lambda^p_G(n).
\]
\end{bthm}

The remainder of the paper is devoted to initial properties of these quasi-isometry invariants organised in direct analogy with the corresponding theory of Neumann-Poincar\'e profiles. We begin with the extremal cases $p=1$ and $p=\infty$.

For $p=\infty$ the Dirichlet-Poincar\'e profile also depends only on growth.

\begin{bprop} Let $X$ be a bounded degree graph.
\[
 D\Lambda^\infty_X(n) \simeq \min\setcon{\frac{m}{\underline{\kappa}(m)}}{m\geq n}
\]
where $\underline{\kappa}(m)=\max\setcon{k}{\exists x:\ \abs{B(x,k)}\leq m}$ is the (lower) inverse growth function.
\end{bprop}

In the same way that $\Lambda^1$ can be expressed in terms of the Cheeger constants of finite graphs, $D\Lambda^1$ is related to the Cheeger constants of infinite graphs.

We recall that the Cheeger constant of an infinite graph is given by
\[
 h(X) = \inf\setcon{\frac{\abs{\partial_X A}}{\abs{A}}}{A\subset VX,\ \abs{A}<\infty}
\]
\begin{bthm}\label{bthm:DPprofamen} Let $X$ be a connected graph of bounded degree. Then the following are equivalent
\begin{itemize}
 \item $h(X)>0$,
 \item $D\Lambda^1_X(n)\simeq n$,
 \item $D\Lambda^p_X(n)\simeq n$ for every $p\in[1,\infty)$.
\end{itemize} 
\end{bthm}
When $X$ is the Cayley graph of a finitely generated group $G$, $h(X)>0$ if and only if $G$ is non-amenable (this is commonly known as F\o lner's criterion). From this we may easily deduce that $D\Lambda^1_X$ is not monotone under quasi-isometric embeddings, since solvable groups of exponential growth admit undistorted free sub-semigroups \cite{CorTes-trees-in-solv}, so there is a quasi-isometric embedding of a $4$-regular tree (a Cayley graph of the non-amenable free group on two generators) into a solvable (and hence amenable) group.

Moreover, one can completely express $D\Lambda^1_X$ in terms of the \textbf{F\o lner function} of $X$:
\[
F_X(n)=\min\colsetcon{\abs{\Gamma}}{\Gamma\leq X,\ \frac{\abs{\partial_X\Gamma}}{\abs{\Gamma}}\leq \frac{1}{n}}.
\]

\begin{bthm}\label{bcor:FGDh1Fol} Let $X$ be a bounded degree graph. Then, for $p\in(1,\infty)$,
\[
 D\Lambda^p_X(n) \lesssim \min\setcon{\frac{m}{(\underline{F}_X(m)})^\frac1p}{m\geq n}
\]
where $\underline{F}_X(n)$ is the inverse F\o lner function $\underline{F}_X(n)=\max\setcon{k}{F(k)\leq n}$. In the case $p=1$ we have the stronger result
\[
 D\Lambda^p_X(n) \simeq \min\setcon{\frac{m}{\underline{F}_X(m)}}{m\geq n}.
\]
\end{bthm}
Since F\o lner functions of amenable groups have been extensively studied there are a number of immediate consequences of this result, some of which we list in \S\ref{sec:Folner}. 

Grigorchuk and Pansu conjecture that the F\o lner function grows either polynomially, or at least exponentially \cite[Conjecture 5(ii)]{Grig_Milnor}. Reinterpreted in terms of Dirichlet-Poincar\'e profiles this is equivalent to the following:

\begin{bconj}\label{conj} Let $G$ be a finitely generated group. Either there is some $d$ such that $D\Lambda^p_G(n)\simeq n^{1-\frac1d}$ for all $p\in[1,\infty)$, or $D\Lambda^p_G(n)\gtrsim \frac{n}{\log(n)}$ for all $p\in[1,\infty)$.
\end{bconj}

Continuing the comparison with Neumann-Poincar\'e profiles, we see that Dirichlet-Poincar\'e profiles are also monotonic with respect to $p\in[1,\infty)$, however the relationship with $D\Lambda^\infty$ is very different:

\begin{bthm}\label{bthm:pleqq} Let $X$ be a connected graph of bounded degree. Then for all $1\leq p\leq q<\infty$:
\[
 D\Lambda^p_X(n)\lesssim D\Lambda^q_X(n).
\]
\end{bthm}

\begin{bthm}\label{bthm:inftyleq1} Let $G$ be a finitely generated group. Then
\[
 D\Lambda^\infty_G(n) \lesssim D\Lambda^1_G(n).
\]
\end{bthm}
Theorem \ref{bthm:inftyleq1} is a direct consequence of Varopoulos' inequality (cf.\ \textup{\cite[Th\'eor\`eme 1]{CoulSalCoste}})
\medskip

We finish with some upper bounds coming from geometric properties of groups. From \cite[Proposition 9.5]{HumeMackTess1} we see that groups with finite linearly-controlled asymptotic dimension (cf. \cite{Ass}) satisfy $\Lambda^p_X(n)\lesssim \Lambda^\infty_X(n)$ for all $p\in[1,\infty)$. For Dirichlet-Poincar\'e profiles this follows from Theorem \ref{bcor:FGDh1Fol} in the case $p=1$ via work of Nowak.

\begin{bcor}\label{bcor:GamenfinAN}\textup{\cite[Theorem 7.1]{Nowak_isop}} Let $G$ be a finitely generated amenable group with finite linearly-controlled asymptotic dimension. Then
\[
 D\Lambda^1_G(n) \simeq \frac{n}{\kappa(n)}
\]
where $\kappa(n)$ is the inverse growth function $\kappa(n)=\min\setcon{k}{\abs{B(1,k)}> n}$.
\end{bcor}

The key examples of groups satisfying the above hypotheses are polycylic groups and wreath products $F\wr\Z$ with $F$ finite. These groups also have \textbf{controlled F\o lner pairs} in the sense of \cite{Tess_asymiso}. For these groups we have the following:

\begin{bthm}\label{bcor:GcontFolpairs} Let $G$ be a finitely generated amenable group with controlled F\o lner pairs. Then, for all $p\in[1,\infty)$
\[
 D\Lambda^p_G(n) \simeq \frac{n}{\underline{\kappa}(n)}
\]
where $\kappa(n)$ is the inverse growth function $\kappa(n)=\min\setcon{k}{\abs{B(1,k)}> n}$.
\end{bthm}

\subsection{Questions}

It is natural to ask for which groups $\Lambda^p_G\simeq D\Lambda^p_G$. For non-amenable groups this will be exceptionally rare, as $\Lambda^p_G(n)\simeq n$ if and only if $G$ contains an expander $(\Gamma_n)_n$ where the $\abs{\Gamma_n}$ grow at most exponentially in $n$. 

Amongst amenable groups this equality appears to be much more common. We know from Theorem \ref{bcor:GcontFolpairs}, \cite{HumeMackTess1} and \cite{HumeMackTess2} that these profiles are equal for all virtually polycylic groups.

However, it is certainly not the case that $\Lambda^1_{G}\simeq D\Lambda^1_G$ holds for amenable groups. By \cite{HumeMackPoorly} there are elementary amenable groups $G_{d}$ such that
\[
 \Lambda^1_{G_{d}}(n) \lesssim \log^d(n) \lesssim \frac{n}{\log(n)} \lesssim D\Lambda^1_{G}(n).
\]
 where $\log^d$ denotes the $d$th iterate of $\log$. Both of the following questions appear to be open.
 
 \begin{question} Does $\Lambda^p_G(n) \lesssim D\Lambda^p_G(n)$ hold for every $p\in[1,\infty]$ and every finitely generated group $G$? 
 \end{question}
 
 \begin{question} Does there exist a finitely generated group $G$ and $1\leq p <q <\infty$ such that $D\Lambda^p_G(n)\not\simeq D\Lambda^q_G(n)$?
 \end{question}
 
%
 
 \subsection{Disclaimer} Some results in this note are likely to be well-known to experts in the area. The goal here is to represent them as profiles in the style of \cite{HumeMackTess1}, and hopefully to provoke new questions.
 
\section{The short answer}

We recall that
\[
 D^*\Lambda^p_X(n):=\sup\colsetcon{\abs{\Gamma}Dh^p_X(\Gamma)}{\Gamma\leq X,\ \partial_X\Gamma\neq  V\Gamma,\ \abs{V\Gamma}\leq n}.
\]

\begin{proposition}\label{prop:detectinf} Let $X$ be an infinite graph of maximal degree $d$. Then, for all $n\geq d+1$,
\[
 D^*\Lambda^p_X(n) =
 \left\{
  \begin{array}{rl}
   (d+1)^\frac1p n & \textup{if } p<\infty, \\
   n & \textup{if } p=\infty.
  \end{array}
 \right.
\]
\end{proposition}
\begin{proof}
Let $v\in VX$ have degree exactly $d$ and let $\Gamma_n\leq X$ satisfy $\abs{\Gamma}=n$ and $\partial_X\Gamma=\set{v}$. Let $f:\Gamma_n\to \R$ be a function such that $f|_{\partial_X \Gamma_n}\equiv 0$. Firstly suppose $p<\infty$.
\begin{equation}\label{eq:normp}
\norm{f}_p = \abs{f(v)} \quad \textup{and} \quad\norm{\nabla f}_p = (d+1)^{\frac{1}{p}} \abs{f(v)}.
\end{equation}
Hence $Dh^p(\Gamma_n)\geq (d+1)^{\frac{1}{p}}$. Next we prove an upper bound on $Dh^p_X(\Gamma)$ for any finite $\Gamma\leq X$ with $V\Gamma\setminus \partial_X\Gamma\neq\emptyset$. Fix a vertex $v\in  V\Gamma\setminus \partial_X\Gamma$ and define $f=1_{\set{v}}$. Then 
\begin{equation}\label{eq:Dhp}
 Dh^p_X(\Gamma)\leq \frac{\norm{\nabla f}_p}{\norm{f}_p} \leq (d+1)^{\frac{1}{p}}\leq d+1.
\end{equation}

If $p=\infty$, then following the above strategy (\ref{eq:normp}) can be replaced by $\norm{f}_\infty = \abs{f(v)} = \norm{\nabla f}_\infty$. Hence $Dh^\infty_X(\Gamma_n) \geq 1$.  Moreover, taking the same function $f=1_{\set{v}}$ (\ref{eq:Dhp}) can be replaced with $Dh^\infty_X(\Gamma) \leq 1$ and thus the result holds when $p=\infty$.
\end{proof}

\section{The long answer}

We now pass to the definition of Dirichlet-Poincar\'e profile we will consider for the remainder of the paper.

\begin{defn}\label{def:DPprofile} Let $X$ be a graph of bounded degree. The $L^p$-\textbf{Dirichlet-Poincar\'e profile} of $X$ is given by
\[
 D\Lambda^p_X(n) = \inf\colsetcon{\abs{\Gamma}Dh^p_X(\Gamma)}{\Gamma\leq X,\ \abs{V\Gamma}\geq n}.
\]
\end{defn}

\subsection{Elementary observations}

We begin with two elementary but useful observations.

\begin{lemma}\label{lem:nesting} Let $X$ be a graph, let $B$ be a finite subgraph of $X$ and $A$ a subgraph of $B$. Then for all $p$, $Dh^p(B)\leq Dh^p(A)$.
\end{lemma}
\begin{proof} For each $f:VA\to[0,\infty)$ satisfying $f|_{\partial_X{A}}\equiv 0$, define $f':VB\to[0,\infty)$ by 
\[
f'(b)= \left\{\begin{array}{rl} f(b) & \textrm{if}\ b\in VA, \\ 0 & \textrm{otherwise}.\end{array}\right.
\] 
It is clear that for every $p$, $\norm{f'}_p=\norm{f}_p$ and $\norm{\nabla f'}_p=\norm{\nabla f}_p$. The result follows.
\end{proof}

Given a finite subgraph $\Gamma\leq X$ and some $a\geq 1$, we define 
\[
Dh^p_a(\Gamma)=\inf\setcon{\frac{\norm{\nabla_a f}}{\norm{f}}}{f|_{\partial_X\Gamma} \equiv 0},
\]
where $\nabla_a f(x)=\sup\setcon{\abs{f(y)-f(y')}}{y,y'\in B_{\Gamma}(x,a)}$.

\begin{lemma}\label{lem:changethickness} Let $\Gamma$ be a finite graph with maximal degree $d$, let $a\geq 1$. There exists a constant $c=c(a,d)>0$ such that
\begin{equation}\label{eq:changethick}
 cDh^p_a(\Gamma)\leq Dh^p(\Gamma) \leq Dh^p_a(\Gamma).
\end{equation}
\end{lemma}
\begin{proof} The right-hand inequality is obvious. Fix $a\geq 1$ and let $B_a$ be the maximal cardinality of a closed ball of radius $a$ in $\Gamma$. For every $x\in \set{\nabla_a f\geq t}$ choose $y,y'\in B_\Gamma(x,a)$ such that $\abs{f(y)-f(y')}\geq t$. Consider geodesics from $y$ to $x$ and from $x$ to $y'$. By the triangle inequality there is an edge $uv$ on one of these geodesics such that $\abs{f(u)-f(v)}\geq \frac{t}{2a}$. It follows that there is some $u\in B_\Gamma(x,a)$ contained in $\set{\nabla f\geq \frac{t}{2a}}$. Hence 
\begin{equation}\label{eq:nabla}
\abs{\set{\nabla_a f\geq t}} \leq B_a\abs{\set{\nabla f\geq \frac{t}{2a}}} \leq d^{a+1}\abs{\set{2a\nabla f\geq t}}.
\end{equation}
Using the co-area formula:
\[
 \sum_{x\in V\Gamma}\abs{g(x)} = \int_{\R^+} \set{g\geq t} dt
\]
and $(\ref{eq:nabla})$ we have
\[
 \sum_{x\in V\Gamma}\abs{\nabla_a f(x)}^p \leq d^{a+1}\sum{\abs{2a\nabla f(x)}^p}\leq (2a)^pd^{a+1}\sum_{x\in V\Gamma}\abs{\nabla f(x)}^p.
\]
Hence $\norm{\nabla_a f}_p \leq 2ad^{\frac{a+1}{p}}\norm{\nabla f}_p$, and $(\ref{eq:changethick})$ follows.
\end{proof}

\subsection{Quasi-isometry invariance}
These Dirichlet-Poincar\'e profiles are quasi--isometry invariants.

\begin{theorem}\label{thm:qiinv} Let $X,Y$ be infinite connected bounded degree graphs and let $q:X\to Y$ be a quasi--isometry. Then for any $p\in[1,\infty)$, $D\Lambda^p_X(n)\simeq D\Lambda^p_Y(n)$.
\end{theorem}

\begin{proof}[Proof of Theorem $\ref{thm:qiinv}$] Let $q:X\to Y$ be a $(K,C)$-quasi-isometry, so for all $x,x'\in VX$,
\[
 K^{-1} d_X(x,x')-C \leq d_Y(q(x),q(x'))\leq Kd_X(x,x') +C,
\]
and for every $y\in VY$ there is some $x$ so that $d_Y(q(x),y)\leq C$. Let $d,d'$ be the maximal vertex degrees of $X,Y$ respectively. We will prove $D\Lambda^p_Y(n)\lesssim D\Lambda^p_X(n)$ by constructing for each finite $\Gamma\leq X$ a graph $\Gamma'\leq Y$ with a comparable number of vertices and $Dh^p_Y(\Gamma')$ bounded from above by a fixed multiple of $Dh^p_X(\Gamma)$.

For each finite $\Gamma\leq X$ we define $\Gamma'$ to be the full subgraph of $Y$ whose vertex set is the closed $C$-neighbourhood of $q(\Gamma)$. Given any $f:V\Gamma\to [0,\infty)$ with $f|_{\partial_X\Gamma}\equiv 0$ we define a comparison function $f':V\Gamma'\to[0,\infty)$. If $y\in \partial_Y \Gamma'$ define $f'(y)=0$, otherwise define $f'(y)=\max\setcon{\abs{f(x)}}{d_Y(y,q(x))\leq C}$.

We first find a lower bound for $\norm{f'}_p$ in terms of $\norm{f}_p$. For every $x\in V\Gamma$, $f'(q(x))\geq f(x)$, and the pre-image of a vertex in $Y$ has diameter at most $KC$. Hence
\begin{equation}\label{eq:controlnorm}
 \norm{f'}^p_p=\sum_{y\in V\Gamma'} f'(y)^p \geq (d+1)^{-KC} \sum_{x\in V\Gamma} f(x)^p = (d+1)^{-KC}\norm{f}^p_p.
\end{equation}
We next find an upper bound for $\norm{\nabla f'}_p$ in terms of $\norm{\nabla f}_p$. For $a\geq 1$ define $\Gamma_a$ to be the full subgraph of $X$ whose vertex set is the closed $a$--neighbourhood of $V\Gamma$ in $X$. Define $f_a:V\Gamma_a\to[0,\infty)$ by $f_a(x)=f(x)$ if $x\in V\Gamma$ and $0$ otherwise. By direct calculation, $\norm{\nabla f_a}_p=\norm{\nabla f}_p$, so by Lemma \ref{lem:changethickness}, for any $a$ there is a constant $L=L(a,d)>0$ such that
\[
 \norm{\nabla_af_a}_p \leq L\norm{\nabla f}_p.
\]
let $y_0\in V\Gamma'$ and choose $y_1$ so that $y_0y_1\in E\Gamma'$ and $\nabla f'(y_0)=\abs{f'(y_0)-f'(y_1)}$. Now we may choose $x_0,x_1\in V\Gamma$ such that $f'(y_j)=f(x_j)$ and $d_Y(q(x_j),y_j)\leq C$. We have $d_Y(q(x_0),q(x_1))\leq 2C+1$ so
\[
 d_X(x_0,x_1)\leq Kd_Y(q(x_0),q(x_1))+KC = 3KC + K.
\]
Choosing $a\geq 3KC+K$ we see that
\[
 \nabla_a f_a(x_0) \geq \abs{f_a(x_0)-f_a(x_1)}=\abs{f'(y_0)-f'(y_1)}=\nabla f'(y_0).
\]
Now the set of $y_0\in V\Gamma'$ for which we may choose a fixed $x_0\in V\Gamma$ is contained in the closed ball of radius $C$ around $q(x_0)$. Thus
\[
 \norm{\nabla f'}_p \leq (d'+1)^C \norm{\nabla_a f_a}_p \leq L(d'+1)^C \norm{\nabla f}_p.
\]
Combining this with ($\ref{eq:controlnorm}$), we have $Dh^p(\Gamma')\leq (d+1)^{KC}L(d'+1)^C Dh^p(\Gamma)$. If $\abs{\Gamma}\geq n$, then $\abs{\Gamma'}\geq (d+1)^{-KC}n$, hence
\[
 D\Lambda^p_Y(n) \leq (d+1)^{KC}L(d'+1)^CD\Lambda^p_X((d+1)^{KC}n),
\]
so $D\Lambda^p_Y(n)\lesssim_{K,C,d,d'} D\Lambda^p_X(n)$. The opposite inequality is obtained by considering a quasi-inverse $r:Y\to X$ of $q$.
\end{proof}

\subsection{The extremal cases $p=1$ and $p=\infty$}

We next explore the extremal cases, starting with $p=\infty$.

\begin{proposition}\label{prop:inftycase} Let $X$ be an infinite graph and let $\Gamma$ be a subgraph of $X$. Define $l_\Gamma$ to be the radius of the largest ball in $X$ which is contained in $\Gamma$. Then $Dh^\infty(\Gamma) =l_\Gamma^{-1}$ and
\[
 D\Lambda^\infty_X(n) \simeq \inf\setcon{\frac{m}{\underline{\kappa}(m)}}{m\geq n}
\]
where $\underline{\kappa}(m)$ is the maximal $k$ such that there is a ball of radius $k$ in $X$ containing at most $m$ vertices.
\end{proposition}
\begin{proof}
Let $\Gamma$ be a finite subgraph of $X$ with $V\Gamma\neq\partial_X\Gamma$ and let $f:V\Gamma\to[0,\infty)$ satisfy $f|_{\partial_X \Gamma}\equiv 0$. Pick $x\in V\Gamma$ such that $f(x)=\norm{f}_\infty$. Let $P$ be a path from $x$ to a vertex in $\partial_X\Gamma$ which has length at most $l_\Gamma\geq 1$. It follows from the triangle inequality that $\norm{\nabla f}_\infty\geq \frac{1}{l_\Gamma}\norm{f}_\infty$, hence
\[
 Dh^\infty(\Gamma)\geq \frac{1}{l_\Gamma}.
\]
Now fix $x$ so that $B_X(x,l_\Gamma)\subseteq \Gamma$ and define $f:V\Gamma\to[0,\infty)$ by $f(y)=\max\set{0, l_\Gamma - d_X(x,y)}$. It is clear that $\norm{f}_\infty = l_\Gamma$ and $\norm{\nabla f}_\infty\leq 1$. Hence $Dh^\infty(\Gamma)\leq \frac{1}{l_\Gamma}$. By definition
\[
 D\Lambda^\infty_X(n)\simeq\inf\setcon{\frac{\abs{\Gamma}}{l_\Gamma}}{\Gamma\leq X,\ l_\Gamma\geq 1,\ \abs{\Gamma}\geq n} \simeq \inf\setcon{\frac{m}{\underline{\kappa}(m)}}{m\geq n}. \qedhere
\]
\end{proof}

Recall that $\Lambda^\infty_X(n) \simeq \sup\setcon{\frac{m}{\overline{\kappa}(m)}}{m\leq n}$ where $\overline{\kappa}$ is the (upper) inverse growth function $\overline{\kappa}(m)=\max{\setcon{k}{\forall x\ \abs{B(x;k)}\leq m}}$. For most groups (certainly those with polynomial or exponential growth), $D\Lambda^\infty_X(n)\simeq \Lambda^\infty_X(n)$.
\medskip

As with Poincar\'e profiles of Neumann type, the $L^1$-Dirichlet-Poincar\'e profile is determined by a combinatorial connectivity constant. We recall that the \textbf{Cheeger constant} of an infinite graph with bounded degree is given by
\[
 h(X) = \inf\setcon{\frac{\abs{\partial_X A}}{\abs{A}}}{A\subset VX,\ \abs{A}<\infty}
\]

\begin{theorem}\label{thm:DPprofamen} Let $X$ be a connected graph of bounded degree. The following are equivalent
\begin{enumerate}
\item $D\Lambda^p_X(n)\not\simeq n$ for every $p\in[1,\infty)$,
\item $D\Lambda^1_X(n)\not \simeq n$,
\item $h(X)=0$.
\end{enumerate} 
\end{theorem}

\begin{proof} $(i)\Rightarrow(ii)$ is immediate. Fix $d$ to be the maximal degree of a vertex in $X$. We start with $(ii)\Rightarrow(iii)$. Suppose $D\Lambda^1_X(n)\not\simeq n$. From (\ref{eq:Dhp}) we have $Dh^1(\Gamma)\leq d+1$ for every finite subgraph $\Gamma$ of $X$. Therefore, there must be a sequence of finite subgraphs $\Gamma_n$ of $X$ such that $Dh^1(\Gamma_n)\leq \frac{1}{n}$.

For each $n$, let $f:V\Gamma_n\to\R$ satisfy $f|_{\partial_X\Gamma_n}\equiv 0$, $f\geq 0$, and
\[
 \frac{\norm{\nabla f}_1}{\norm{f}_1} \leq \frac{2}{n}.
\]
Using the co-area formula (cf.~\cite[Proposition 6.6]{HumeMackTess1})
\[
 \norm{\nabla f}_1 = \int_{\R^+} \abs{\partial_X\set{f>t}} dt,
\]
and $\norm{f}_1 = \int_{\R^+} \abs{\set{f>t}} dt$, we see that there is some $t>0$ such that $S_t=\set{f>t}\subset V\Gamma$ satisfies
\[
 \frac{\abs{\partial_X S_t}}{\abs{S_t}} \leq \frac{2}{n}.
\]
Thus $h(X)=0$.

Finally, we show $(iii)\Rightarrow(i)$. Suppose $h(X)=0$. If $X$ is finite there is nothing to prove, so assume it is infinite. There is a family of finite subgraphs $\Gamma_n$ ($n\geq 2$) of $X$ such that
\begin{equation}\label{eq:Folner}
 \frac{\abs{\partial_X \Gamma_n}}{\abs{\Gamma_n}}\leq \frac{1}{n}.
\end{equation}
Since $X$ is infinite and connected, $\abs{\partial_X \Gamma_n}\geq 1$, so $\abs{\Gamma_n}\geq n$. Let $f_n$ be the characteristic function of the set $V\Gamma_n\setminus\partial_X \Gamma_n$. By construction
\[
 \norm{f_n}_p = \left(\abs{V\Gamma_n\setminus\partial_X \Gamma_n}\right)^\frac1p\geq \left(\frac{n-1}{n}\abs{\Gamma_n}\right)^\frac1p,
\]
and $\nabla f_n$ is the characteristic function of the set of vertices in $\Gamma_n$ at distance $\leq 1$ from $\partial_X \Gamma_n$. Hence
\[
 Dh^p(\Gamma_n) \leq \frac{\norm{\nabla f_n}_p}{\norm{f_n}_p}\leq \left((d+1)\abs{\partial_X \Gamma_n}\frac{n}{(n-1)\abs{\Gamma_n}}\right)^\frac1p \leq \left(\frac{d+1}{n-1}\right)^\frac1p
\]
where the last step uses (\ref{eq:Folner}). It follows that $D\Lambda^1_X(n)\not\simeq n$, since for every $n$,
\begin{equation}\label{eq:folfunc}
 D\Lambda^p_X(\abs{\Gamma_n}) \leq \abs{\Gamma_n}Dh^p(\Gamma_n) \leq \abs{\Gamma_n}\left(\frac{d+1}{n-1}\right)^\frac1p.
\end{equation}
\end{proof}

\begin{corollary}\label{cor:amenp} Let $X$ be a graph of bounded degree. Then $D\Lambda^p_X(n)\not\simeq n$ for every $p\in[1,\infty)$ if and only if $h(X)>0$.
\end{corollary}
\begin{proof} The forward implication is immediate from Theorem \ref{thm:DPprofamen}, the reverse implication follows from monotonicity (Proposition \ref{prop:DPmonotonep}) and Theorem \ref{thm:DPprofamen}.
\end{proof}

Recall that for a graph $X$ satisfying $h(X)=0$ the \textbf{F\o lner function} is: \cite{Vershik_Amen_Approx}
\[
 F(n)=\min\setcon{k}{\frac{\abs{\partial_X\Gamma}}{\abs{\Gamma}}\leq\frac{1}{n} \textup{ for some }\Gamma\leq X\textup{ with } \abs{\Gamma}=k}.
\]
\begin{corollary}\label{cor:folfuncsuperlin} Let $X$ be a graph of bounded degree such that $h(X)=0$. Then for $p\in(1,\infty)$
\[
 D\Lambda^p_X(n)\lesssim \inf\setcon{\frac{m}{(\underline{F}_X(m))^\frac1p}}{m\geq n}
\]
where $\underline{F}_X(m)=\max\setcon{k}{F(k)\leq m}$. In the case $p=1$ we have 
\[
 D\Lambda^1_X(n)\simeq \inf\setcon{\frac{m}{\underline{F}_X(m)}}{m\geq n}.
\]
\end{corollary}
\begin{proof}
For the upper bound fix $m$, and let $\Gamma'$ be a subgraph of $X$ satisfying
\[
 \abs{\Gamma'}\leq m \quad \textup{and} \quad \frac{1}{\underline{F}_X(m)+1} <\frac{\abs{\partial_X\Gamma}}{\abs{\Gamma}} \leq \frac{1}{\underline{F}_X(m)}.
\]
From the proof of Theorem \ref{thm:DPprofamen} $(iii)\Rightarrow(i)$, we have that
\[
 Dh^p_X(\Gamma') \leq \left(\frac{d+1}{\underline{F}_X(m)-1}\right)^\frac1p
\]
Hence, for any $m$-vertex subgraph $\Gamma$ of $X$ containing $\Gamma$ we have
\[
 Dh^p_X(\Gamma) \leq \left(\frac{d+1}{\underline{F}_X(m)-1}\right)^\frac1p
\]
by Lemma \ref{lem:nesting}. Thus
\[
 D\Lambda^p_X(m) \leq m\left(\frac{d+1}{\underline{F}_X(m)-1}\right)^\frac1p
\]
as required. For the lower bound, let $\Gamma$ be a finite subgraph of $X$.
Fix $k$ maximal such that $Dh^1(\Gamma)\leq \frac{1}{k}$, so $\abs{\Gamma}\geq F(k)$. If $F(k)<n$ there is nothing to prove. If $F(k)\geq n$, then by assumption
\[
 \abs{\Gamma}Dh^1(\Gamma) \geq \frac{F(k)}{k+1} \geq \frac{F(k)}{2k} \gtrsim\inf\setcon{\frac{m}{\underline{F}_X(m)}}{m\geq n},
\]
as required.
\end{proof}

\subsubsection{Consequences for the F\o lner function}\label{sec:Folner}

We give three consequences of Corollary \ref{cor:folfuncsuperlin} for finitely generated groups.

\begin{theorem}\textup{\cite[Theorem 7.1]{Nowak_isop}}
Let $G$ be a finitely generated amenable group with finite linearly controlled asymptotic dimension. Then
\[
 D\Lambda^1_G(n) \simeq \inf\setcon{\frac{m}{\kappa(m)}}{m\geq n},
\]
where $\kappa$ is the inverse growth function of $G$.
\end{theorem}

\begin{theorem}\label{thm:itwreath}\textup{\cite{Ersh_isoprofs}} We have the following:
\begin{enumerate}
 \item For $G=\Z\wr\Z$, $D\Lambda^1_G(n)\simeq \frac{n\log\log n}{\log n}$.
 \item For $G=\Z_2\wr \Z^d$, $D\Lambda^1_G(n)\simeq \frac{n}{(\log n)^{1/d}}$.
 \item For $G=\Z\wr(\Z\wr(\Z\ldots(\Z\wr\Z))\ldots)$ where $\Z$ occurs $k$ times, we have $D\Lambda^1_G(n)\simeq n\left(\frac{\log\log n}{\log n}\right)^{1/k}$.
  \item For $G=((\ldots((\Z\wr\Z)\wr\Z)\ldots\wr\Z)$ where $\Z$ occurs $k$ times, we have $D\Lambda^1_G(n)\simeq \frac{n}{\phi^k n}$, where $\phi^k$ is the $k-1$-fold iteration of $\log$ divided by the $k$-fold iteration of $\log$.
\end{enumerate}
\end{theorem}

\begin{theorem}\textup{\cite{Ers06}}
For every function $f : \N \to \N$ such that $\lim_{n\to\infty}\frac{f(n)}{n}=0$, there is a finitely generated group of intermediate growth such that
\[
 D\Lambda^1_G(n)\gtrsim f(n).
\]
\end{theorem}

\subsection{Dependence on $p$}

Dirichlet-Poincar\'e profiles satisfy many of the same properties as the Poincar\'{e} profile $\Lambda_X^p$ such as monotonicity:

\begin{proposition}\label{prop:DPmonotonep} Let $X$ be a graph of bounded degree. For every $1\leq p\leq q <\infty$ there is a constant $C=C(p,q)$ such that
\[
 D\Lambda^p_X(n)\leq CD\Lambda^q_X(n).
\]
\end{proposition}
\begin{proof}
Let $d$ be the maximal vertex degree of $X$. Choose $\Gamma\leq X$ with $n\leq \abs{\Gamma}<\infty$ and $g:V\Gamma\to[0,\infty)$ such that $g|_{\partial_X\Gamma}\equiv 0$ and
\[
 \abs{\Gamma}\frac{\norm{\nabla g}_q}{\norm{g}_q} \leq 2D\Lambda^q_X(n).
\]
Define $f:V\Gamma\to[0,\infty)$ by $f(v)=g(v)^{q/p}$. Now $\norm{f}^p_p= \norm{g}^q_q$. By (\ref{eq:Dhp}) we need only consider functions $g$ such that $\norm{\nabla_a g}_q \leq (d+1)^\frac1q\norm{g}_q$.

By the mean value theorem (see e.g.\ Matou\v{s}ek~\cite[Lemma 4]{Mat-97-ExpandersLp}), for every $s,t \in \R$ and $\alpha \geq 1$,
\[ |\{s\}^\alpha-\{t\}^\alpha| \leq \alpha (|s|^{\alpha-1}+|t|^{\alpha-1}) |s-t|. \]
For each $v\in V\Gamma$ we apply this to $s=g(v)$, $t=g(w)$, $\alpha=\frac{q}{p}$ for each edge $vw\in E\Gamma$ to see that
\[
 \nabla f(v) \leq\frac{2q}{p} g_1(v)^{\frac{q-p}{p}}\nabla g(v)
\]
where $g_1(v)=\max\setcon{\abs{g(w)}}{d_\Gamma(v,w)\leq 1}$. By definition $g_1(v) \leq g(v) + \nabla_g(v)$. Now
\begin{align*}
 & \norm{g}^q_q Dh^p_a(\Gamma)^p 
   = \norm{f}_p^p Dh^p_a(\Gamma)^p
 \\
 &
 \leq \sum_{v\in V\Gamma} \nabla f(v)^p 
 \\ &
 \leq \left(\frac{2q}{p}\right)^p \sum_{v\in V\Gamma} \left(\abs{g(v)}+\nabla g(v)\right)^{q-p} \nabla g(v)^p
 \\
 &
 \overset{(\star)}{\leq} \left(\frac{2q}{p}\right)^p 2^{q-p} \left( \sum_{v\in V\Gamma} \abs{g(v)}^{q-p}\nabla g(v)^p  + \norm{\nabla g}_q^q \right)
 \\ &
 \overset{(\dagger)}{\leq} \frac{2^q q^p}{p^p} \left( \norm{g}_q^{q-p} \norm{\nabla g}_q^p+ (d+1)^{\frac{q-p}{q}} \norm{g}_q^{q-p} \norm{\nabla g}_q^p \right)
 \\ &
 \preceq_{p,q} \norm{g}_q^{q-p} \norm{\nabla g}_q^p,
\end{align*}
where $(\star)$ follows from $(s+t)^\alpha \leq 2^\alpha(s^\alpha+t^\alpha)$ for any $s,t,\alpha>0$, and $(\dagger)$ follows from H\"older's inequality and $\norm{\nabla_a g}_q \leq (d+1)^\frac1q\norm{g}_q$. Rearranging and taking $p$th roots, we see that
\[
 Dh^p(\Gamma) \preceq_{p,q} \frac{\norm{\nabla g}_q}{\norm{g}_q}.
\]
\end{proof}

The relationship between $D\Lambda^\infty$ and $D\Lambda^1$ is a well-known inequality.

\begin{proposition}\label{prop:DPmonotoneinfty}\textup{\cite[Th\'eor\`eme 1]{CoulSalCoste}} Let $X$ be a graph of bounded degree satisfying the pseudo-Poincar\'e inequality (for example Cayley graphs of finitely generated groups)
\[
 \norm{f-f_r}_1\leq Cr\norm{\nabla f}_1
\]
where $f_r(x)=\abs{B(x,r)}^{-1}\sum_{v\in B(x,r)} f(v)$. Then $D\Lambda^\infty_X(n)\lesssim D\Lambda^1_X(n)$.
\end{proposition}

\subsection{Monotonicity with respect to subgroups}

\begin{theorem} Let $H$ be a finitely generated subgroup of a finitely generated group $G$. Then, for every $p\in[1,\infty]$,
\[
 D\Lambda^p_H(n) \lesssim  D\Lambda^p_G(n).
\]
\end{theorem}
\begin{proof} For $p=\infty$, this follows from Proposition \ref{prop:inftycase}. The rest of the proof is adapted from \cite[Theorem 18.100]{DK18}.

Let $S\subset T$ be finite symmetric generating sets for $H$ and $G$ respectively. Let $X=Cay(H,S)$ and $Y=Cay(G,T)$.

Let $\Gamma$ be a finite subgraph of $Y$ with $m$ vertices. Now $\Gamma$ intersects finitely many cosets of $H$ which we label $g_1H,\ldots,g_kH$ (note $k$ will depend on $\Gamma$). Denote $\Gamma_i=g_i^{-1}(\Gamma\cap g_iY)$ considered as a subgraph of $Y$. For each function $f:\Gamma\to [0,\infty)$ let $f_i:\Gamma_i\to[0,\infty)$ be defined by $f_i(x)=f(g_i(x))$. Now
\[
 \sum_{i=1}^k \norm{\nabla^{Y} f_i}_p^p \leq \norm{\nabla^X f}_p^p = \epsilon^p_{f,p} \norm{f}_p^p = \epsilon^p_{f,p}\sum_{i=1}^k \norm{f_i}_p^p,
\]
for some $\epsilon_{f,p}$. Therefore, there is some $i$ such that
\[
 \norm{\nabla^{Y} f_i}_p \leq \epsilon_{f,p}\norm{f_i}_p.
\]

Now if $f|_{\partial_Y(\Gamma)}\equiv 0$ then $f_i|_{\partial_X(\Gamma_i)}\equiv 0$, so $Dh^p_X(\Gamma_i)\leq \epsilon_{f,p}$. 

It is immediate that $\abs{\Gamma_i}\leq m$, by Lemma \ref{lem:nesting} we see that $Dh^p_X(\Gamma')\leq \epsilon_{f,p}$ for any $\Gamma'\leq Y$ which contains $\Gamma_i$.

Since, by definition, $Dh^p_Y(\Gamma)=\inf\setcon{\epsilon_{f,p}}{f:V\Gamma\to[0,\infty),\ f|_{\partial_Y(\Gamma)}\equiv 0}$ we have that for every $\Gamma\leq Y$ with $\abs{\Gamma}=m$ there is some $\Gamma'\leq X$ with $\abs{\Gamma'}=m$ and $Dh^p_X(\Gamma')\leq Dh^p_Y(\Gamma)$. Thus
\[
 D\Lambda^p_X(n) \leq  D\Lambda^p_Y(n).\qedhere
\]
\end{proof}

\subsection{Controlled F\o lner pairs}

Let us recall the definition.

\begin{defn}\cite[Definition $4.8$]{Tess_asymiso}
Let $X$ be a graph of bounded degree. We say a family of pairs of finite subsets of $VX$ $(H_m,H_m')_{m\in\N}$ is a \textbf{controlled sequence of F\o lner pairs} if there exists a constant $C\geq 1$ such that
\begin{itemize}
 \item $N_m(H_m)=\setcon{x\in VX}{d(x,H_m)\leq m} \subseteq H_m'$,
 \item $\abs{H_m'}\leq C\abs{H_m}$,
 \item $\diam(H_m')\leq Cm$.
\end{itemize}
\end{defn}

\begin{proposition} Let $G$ be a finitely generated group which admits a controlled sequence of F\o lner pairs. Then for all $p\in[1,\infty]$,
\begin{equation}\label{eq:p-cogrowth}
  D\Lambda_G^p(n)\lesssim \frac{n}{\kappa(n)}.
\end{equation}
\end{proposition}
\begin{proof} For each $m$ consider the function $f:H_m'\to[0,\infty)$ given by $f(v)=\max\set{0,m-d_X(v,H_m)}$. It is clear that $f|_{\partial_X H_m'}\equiv 0$, and (for $p\in[1,\infty)$)
\[
 Dh^p(H'_m)\leq \frac{\norm{\nabla f}_p}{\norm{f}_p} \leq \frac{2\abs{H'_m}^{\frac{1}{p}}}{\abs{H_m}^{\frac{1}{p}}m} \leq \frac{2C^{\frac{1}{p}}}{m}.
\]
while for $p=\infty$ we immediately have $Dh^\infty(H'_m)\leq\frac{2}{m}$.

Now for each $n$ choose $m$ maximal so that $\abs{H'_m}\leq n$ and let $\Gamma$ be any $n$-vertex subgraph of $X$ such that $H'_m\subseteq V\Gamma$. By Lemma \ref{lem:nesting}, $Dh^p(\Gamma) \leq \frac{2C^{\frac{1}{p}}}{m}$, hence
\[
 D\Lambda^p_X(n) \leq \frac{2C^{\frac{1}{p}}n}{m} \lesssim \frac{n}{\kappa(n)}. 
\]
Let us justify the final inequality. Define $b_m=\abs{B(1,m)}$. Now $b_m\leq \abs{H'_m}\leq n < \abs{H'_{m+1}} \leq b_{C(m+1)}$. Hence $m\leq \kappa(n) \leq C(m+1)$.
\end{proof}

\end{document}